\documentclass[11pt]{amsart}
\usepackage{amssymb,amsmath, graphicx, verbatim,epsfig,color,mathrsfs}
\usepackage[latin1]{inputenc}
 \usepackage[all,cmtip]{xy}

\newtheorem{lm}{Lemma}
\newtheorem{pr}{Proposition}
\newtheorem{tm}{Theorem}
\newtheorem{cor}{Corollary}

\newtheorem{define}{Definition}
\newtheorem*{example}{Example}

\newcounter{remarkcount}
  {
 }

\title{Virtual normalization and virtual fundamental classes}
\author{Alberto L\'opez Mart\'in}
\thanks{The author was supported in part by the Swiss National Science Foundation project 200020\_126756.}
\address{Institut f\"ur Mathematik, Universit\"at Z\"urich-Irchel,
Z\"urich, CH-8057}
\email{alopez@math.uzh.ch}

\begin{document}
\date{\today}
\bibliographystyle{plain}
\begin{abstract}
  In this paper, we compare the virtual fundamental classes of the stacks of $(g,\beta,\mu)$-stable ramified maps $\overline{\mathfrak U}_{g,\mu}(X,\beta)$ and of $(g,\beta,\mu)$-log stable ramified maps $\overline{\mathfrak U}^{\mathrm{log}}_{g,\mu}(X,\beta)$. For that we will see how they are identified via virtual normalization and then apply Costello's push-forward formula. \end{abstract}
  \maketitle
\setcounter{section}{-1}
\section{Introduction}
\label{intro}
An important aspect of moduli spaces in algebraic geometry is the notion of \emph{virtual fundamental class}.

Inspired by work of Fukaya and Oh in the symplectic case, B.\,Kim, \linebreak A.\,Kresch, and Y.G.\,Oh \cite{KKO} presented a new compactification of the space of maps from pointed nonsingular projective stable curves to a nonsingular complex projective variety with given ramification indices at the marked points.
 
B.\,Kim \cite{Kim} constructed the moduli space of log stable maps with given ramification indices. Each of these spaces is equipped with a virtual fundamental class. In this article we compare the virtual fundamental classes, and show, in particular, that they determine the same numerical invariants.

In Section \ref{one} we introduce the notion of virtual normalization of an algebraic stack $X$. When $X$ has a relative perfect obstruction theory over an algebraic stack $\mathcal{S}$, it may be described as virtually smooth over $\mathcal{S}$. Since normalization commutes with smooth base change,
the base change by a normalization of $\mathcal{S}$ makes sense as virtual
normalization of $X$. That will be the setting in Section \ref{two}. We will make the observation that one variant of the moduli space introduced by Kim can be identified with the virtual normalization of the space of maps constructed by Kim, Kresch, and Oh. This will let us apply K.\,Costello's Theorem 5.0.1 in \cite{Costello}.

\subsubsection*{Acknowledgements.} The author thanks M.\,Olsson for communicating the argument of the proof of Proposition 1. He would also like to thank A.\,Kresch, for his constant guidance and support.

\section{Virtual normalization}
\label{one}

We will assume the reader is familiar with basic concepts (cf. \cite{KKato}) from logarithmic geometry, and adopt the notation in \cite{OlLogGeo} unless otherwise stated.

Let $S$ denote a normal locally noetherian base scheme.
Throughout the paper, schemes will be locally of finite type over $S$, and 
algebraic stacks will be locally of finite type over $S$
with finite-type though not necessarily separated diagonal. All monoids will be commutative. 

Notation: We will use $P\subset P^{sat}\subset P^{gp}$ for an integral monoid $P$.
As in \cite{OlLogGeo}, for $P$ a fine monoid (i.e. integral and finitely generated),
$\mathcal{S}_P$ will denote the
algebraic stack $[\mbox{Spec}(S[P])/\mbox{Spec}(S[P^{gp}])]$.

\begin{lm}\label{lemma}
For any fine monoid $P$ the natural morphism
\begin{equation*}
\mathcal{S}_{P^{sat}}\to \mathcal{S}_P
\end{equation*}
is a normalization.
\end{lm}

\begin{proof}
This is immediate from the (standard) fact that
$\mbox{Spec}(S[Q])$ is normal for any saturated torsion-free fine monoid $Q$.
\end{proof}

Let us introduce some further notation: we will denote by $\mathcal{L}og$ the category of schemes with fine log structures treated in \cite{OlLogGeo}, and as in \cite{OlLogCplx} 
$\mathcal{L}^1$ will denote the stack whose fiber over a scheme $T$
is the category of pairs of fine log structures on $T$ with a morphism
between them. Furthermore,  $\widehat{\mathcal{L}og}$ will denote the substack of
$\mathcal{L}^1$ where the morphism of log structures is saturation.

\begin{pr}
The natural forgetful map (forgetting the saturation)
\begin{equation*}
\widehat{\mathcal{L}og}\to \mathcal{L}og
\end{equation*}
is a normalization.
\end{pr}

\begin{proof} 
By the cartesian diagram \begin{equation*}
\xymatrix{
\mathcal{S}_{P^{sat}}\ar[r]\ar[d] &\widehat{\mathcal{L}og} \ar[d] \\
 \mathcal{S}_{P} \ar[r] & \mathcal{L}og
}
\end{equation*} the result follows from Lemma \ref{lemma}.  

The diagram above is clearly commutative, and by Corollary 5.25 of \cite{OlLogGeo} both horizontal arrows are \'etale. Therefore, to prove it is cartesian we are reduced to checking a bijection between $k$-points (where $k$ is an algebraically closed field). By Lemma 2.15 of \cite{FKato} we have this bijection.
\end{proof}

\begin{define}\label{vnormal}
Let $X$ and $\mathcal{S}$ be algebraic stacks
and $X\to \mathcal{S}$ a morphism. Let $\widehat{\mathcal{S}}\to \mathcal{S}$ be a normalization.
Then the \emph{virtual normalization} of $X$ with respect to $\mathcal{S}$ is the fiber product $X\times_{\mathcal{S}}\widehat{\mathcal{S}}$.
\end{define}

\begin{example}
When $X$ is smooth over $\mathcal{S}$, the normalization coincides with the virtual normalization. The morphism 
$$\phi: \mathscr{B}^{\mathrm{bal}}_{g,n}(\mathscr{S}
_d)\rightarrow \mathscr{A}dm_{g,n,d}
$$
of \cite[Prop. 4.2.2]{ACV03}, which is a normalization, is therefore also a virtual normalization for the morphism from $\mathscr{A}dm_{g,n,d}$ to $\mathcal{L}og$ described in \cite[\S 3B]{Moc95}. 

\end{example}

In the next result, virtual normalization is again taken with respect to the
stack $\mathcal{L}og$.
The fiber products in question are described in
Lemma 2.15 and Corollary 2.16 of \cite{FKato},
respectively Proposition 2.7 of \cite{KKato}.

\begin{cor}
The fiber product of a pair of morphisms in the category of schemes with
fs log structures is the virtual normalization of the fiber product of the
same morphisms in the category of schemes with fine log structures.
\end{cor}

\section{Stable ramified and log stable ramified map spaces}
\label{two}

Our goal in Section \ref{three} will be to compare the virtual fundamental class of the stack of $(g,\beta,\mu)$-stable ramified maps $\overline{\mathfrak U}_{g,\mu}(X,\beta)$ with that of the stack of $(g,\beta,\mu)$-log stable ramified maps $\overline{\mathfrak U}^{\mathrm{log}}_{g,\mu}(X,\beta)$, constructed using the machinery described by B.\ Kim in \cite{Kim}. Let us now recall some definitions.

Let $X$ be a nonsingular projective variety over $k$. In \cite{KKO}, B.\,Kim, A.\,Kresch, and Y.G.\,Oh gave a definition for $(g,\beta,\mu)$-stable ramified maps: $n$--pointed, genus $g$ stable maps to a Fulton--MacPherson (FM) degeneration space with prescribed ramification indices at the marked points. We refer the reader to \cite{FM94} for foundational material on FM degeneration spaces, to \cite{FP97} for an introduction to stable maps, and to \cite{KKO} for the notation not made explicit here.

\begin{define}[\cite{KKO}] Let $\beta\in NE_1(X)$ be an element in the semigroup of effective curve classes. Let $g,n\in\mathbb{Z}_{\geq 0}$, and $\overline\mu = (\mu_1,\dots,\mu_n)$, $\mu_i\in\mathbb{Z}_{\geq 1}$.

A triple 
$$
((\pi:C\rightarrow S,p_1,\dots,p_n), (\pi_{W/S}:W\rightarrow S,\pi_{W/X}:W\rightarrow X), f:C\rightarrow W)
$$
is called a \emph{stable map with $\overline\mu$-ramification from a $n$--pointed, genus $g$ curve to a FM degeneration space $W$ of $X$, representing class $\beta$} if:

\begin{enumerate}
\item $(C,p_1,\dots,p_n)$ is a $n$--pointed, genus $g$, prestable curve over $k$,
\item $  (\pi_{W/S}:W\rightarrow S,\pi_{W/X}:W\rightarrow X) $ is a FM degeneration of $X$ over $S$,
\item $f:C\rightarrow W$ is a morphism over $S$,
\item over each geometric point of $S$ the pushforward of the fundamental class of $C$ is $\beta$,
\item the following four conditions are satisfied:
\begin{enumerate}
\item \emph{Prescribed Ramification Index Condition:}
\begin{itemize}
\item $C^{sm}=f^{-1}(W^{sm})$,
\item $f|_{C^{sm}}$ is unramified everywhere possibly at $p_i$,
\item $f$ has ramification index $\mu_i$ at $p_i$.
\end{itemize}
\item \emph{Distinct Points Condition:} $f(p_i)$, $i=1,\dots n$ are pairwise distinct points of $W$, over each geometric point of $S$.
\item \emph{Admissibility Condition:} For any geometric point $t$ of $S$, if $p$ is a nodal point of $C_t$, then $\widehat{f}^*$ can be expressed as:
$$
\widehat{\mathcal{O}}_{f(p)}\cong\widehat{\mathcal{O}}_{\pi_S(p)}[[z_1,\dots,z_{r+1}]]/(z_1z_2-s)\stackrel{\widehat{f}^*}{\rightarrow}\widehat{\mathcal{O}}_{p}\cong\widehat{\mathcal{O}}_{\pi_S(p)}[[x,y]]/(xy-s')
$$
with $\widehat{f}^*z_1=\alpha_1 x^m$, $\widehat{f}^*z_2=\alpha_2 y^m$, no restriction on $z_i$ for $i>2$, $\alpha_i$ being units in $\widehat{\mathcal{O}}_p$ with condition $\alpha_1\alpha_2\in\widehat{\mathcal{O}}_{\pi_S(p)}$, $s,s'\in\widehat{\mathcal{O}}_{\pi_S(p)}$, and $m$ being a positive integer.
\item \emph{Stability Condition:} over each geometric point of $S$, the following are satisfied:
\begin{itemize}
\item For each ruled component $W_r$ of $W$, there is an \linebreak image of a marking in $W_r$ ($f(p_i)\in W_r$ for some $i$), or a non-fiber image $f(D)\subset W_r$ of an irreducible component $D$ of $C$,
\item For each end component $W_e\cong\mathbb{P}^r$ of $W$, there are either images of two distinct markings in $W_e$ or a non-line image $f(D)\subset W_e$ of an irreducible component $D$ of $C$.

\end{itemize}
\end{enumerate}
\end{enumerate}
\end{define}

The stack of $(g,\beta,\mu)$-stable ramified maps to FM degeneration spaces of $X$ is denoted by $\overline{\mathfrak U}_{g,\mu}(X,\beta)$. They proved it to be a Deligne--Mumford stack of finite type carrying a perfect obstruction theory:
$$R\pi_*(f^*T^\dag_W(-\sum\mu_i p_i))^\vee \rightarrow L^\bullet_{\overline{\mathfrak U}_{g,\mu}(X,\beta)/\mathfrak B}$$ over the stack $\mathfrak B$ of $n$--pointed, genus $g$, prestable curves, FM spaces (of $X$, with $n$-tuples of smooth pairwise distinct points), fine log structures, and pairs of morphisms of log structures:
$$ (C\rightarrow S, W\rightarrow S, N, N^{C/S}\rightarrow N, N^{W/S}\rightarrow N).$$
By the results of Behrend and Fantechi  \cite{BF97}, this perfect obstruction theory yields a virtual fundamental class  $[\overline{\mathfrak U}_{g,\mu}(X,\beta)]^{\mathrm{vir}}$.

Now, let us introduce the definition of the second space of interest to us, as well as the perfect obstruction theory carried by it. In \cite[5.2.6.]{Kim}, Kim defines \emph{log stable $\mu$-ramified maps}. From this construction, we now introduce the notion of \emph{$(g,\beta,\mu)$-log stable ramified map}.

\begin{define}
Let $\beta\in NE_1(X)$. Let $g,n\in\mathbb{Z}_{\geq 0}$, and $\overline\mu = (\mu_1,\dots,\mu_n)$, $\mu_i\in\mathbb{Z}_{\geq 1}$. A pair
$$((f:(C,M_C,p_1,\dots,p_n)\rightarrow (W,M_W))/(S,N), \pi_{W/X}:W\rightarrow X)
$$
is called a \emph{$(g,\beta,\mu)$-log stable ramified map} if:

\begin{enumerate}
\item $((C,M)/(S,N), p_1,\dots,p_n)$ is a $n$-pointed, genus $g$, minimal log prestable curve,
\item $(W,M_W)/(S,N)$ is a log twisted FM type space,
\item $f:(C,M_C)\rightarrow(W,M_W)$ is a log morphism over $(S,N)$,
\item $\underline{f}: C\rightarrow W$, the underlying map to $f$,  is a stable $(g,\beta,\mu)$-ramified map over $S$.

\vspace{2mm}

\end{enumerate}
\end{define}

As mentioned before we use the notation $\overline{\mathfrak U}^{\mathrm{log}}_{g,\mu}(X,\beta)$ for the stack of such maps. By Kim's construction, this stack also carries a perfect obstruction theory:
$$
(R\pi_*f^* T^\dag_{\mathfrak X^+/\mathfrak X} (-\sum \mu_i p_i))^{\vee}\rightarrow  L^{\bullet}_{\overline{\mathfrak U}^{\mathrm{log}}_{g,\mu}(X,\beta)/\mathfrak{MB}}
$$ over $\mathfrak{MB}:=\mathfrak{M}_{g,n}^{\mathrm{log}} \times_{\mathcal{L}og}\, {\mathfrak X}^{+(n),\mathrm{tw}}$, giving rise to a virtual fundamental class  $[\overline{\mathfrak U}^{\mathrm{log}}_{g,\mu}(X,\beta) ]^{\mathrm{vir}}$. Here $\mathfrak{M}_{g,n}^{\mathrm{log}}$ is the algebraic stack in \cite[6.2.1.]{Kim} and, following \cite{KKO}, $\mathfrak{X}$ denotes the stack of FM degeneration spaces, with universal FM space $\mathfrak{X}^+ \rightarrow \mathfrak{X}$, we have the space
$\mathfrak{X}^{+(n)}$, open in the $n$-fold fiber product of $\mathfrak{X}^+$ over $\mathfrak{X}$ of
pairwise distinct $n$-tuples of smooth points, and the tw denotes log twisted FM type space as in \cite{Kim}.
(Because all nodes are distinguished, we may use log twisted rather than extended log twisted FM type spaces here.) So $\mathfrak{MB}$ is an open substack of $\mathfrak B$.

\begin{pr} Let $N'$ be the log structure on $\overline{\mathfrak U}^{\mathrm{log}}_{g,\mu}(X,\beta)$ whose sheaf of monoids is the subsheaf of $N^{C/S}$ generated by $m \log(s')$ for every node of $C$. Then, the commuting diagram 
$$
\xymatrix{
(\overline{\mathfrak U}^{\mathrm{log}}_{g,\mu}(X,\beta), N) \ar[r] \ar[d] & (\overline{\mathfrak U}_{g,\mu}(X,\beta), N^{W/S}) \ar[d] \\
( \overline{\mathfrak U}_{g,\mu}(X,\beta), N^{C/S} ) \ar[r]  &  (\overline{\mathfrak U}_{g,\mu}(X,\beta), N')
}
$$
is cartesian in the category of schemes with fs log structures.

\end{pr}

\begin{proof}
This follows directly from the definition of log prestable map in \cite[5.2.2.]{Kim}.
\end{proof}

\section{Comparison of virtual fundamental classes}
\label{three}

Let us define $\widehat{\mathfrak{B}}= {\mathfrak{B}}\times_{{\mathcal L}og} \widehat{\mathcal{L}og}$. Then, we have the cartesian diagram:

\begin{equation*}
\xymatrix{
\overline{\mathfrak U}^{\mathrm{log}}_{g,\mu}(X,\beta) \ar[r] \ar[d]_{\phi} &
\widehat{\mathfrak{B}} \ar[r] \ar[d] &
\widehat{\mathcal{L}og} \ar[d] \\
\overline{\mathfrak{U}}_{g,\mu}(X,\beta) \ar[r] &
\mathfrak{B} \ar[r] &
\mathcal{L}og
}
\end{equation*}
where $\phi$ is the forgetful map. (The outer square is
cartesian by Corollary 1 and Proposition 2.) From $\widehat{\mathfrak{B}}$ to $ {\mathfrak{B}}$ there is another morphism, which forgets the not-necessarily-saturated log structure, and this morphism is relatively Deligne-Mumford and \'etale by \cite[Prop. 2.11]{OlLogCplx}.
It follows that $L^{\bullet}_{\overline{\mathfrak U}^{\mathrm{log}}_{g,\mu}(X,\beta)/\mathfrak{MB}}=L^{\bullet}_{\overline{\mathfrak U}^{\mathrm{log}}_{g,\mu}(X,\beta)/\widehat{\mathfrak B}}$.

\begin{tm}
We have the equality of virtual fundamental classes $$[\overline{\mathfrak{U}}_{g,\mu}(X,\beta)]^{\mathrm{vir}}=
\phi_*[\overline{\mathfrak U}^{\mathrm{log}}_{g,\mu}(X,\beta) ]^{\mathrm{vir}}.$$
\end{tm}

\begin{proof}
 By the above cartesian diagram, this follows from \cite[Theorem 5.0.1]{Costello}, as normalizations have generic degree 1 and the compatibility of obstruction theories follows from the equality of cotangent complexes above.

\end{proof}

\end{document}